\documentclass[preprint,12pt]{elsarticle}

\usepackage{amssymb}

\newtheorem{thm}{Theorem}

\newdefinition{dfn}{Definition}
\newdefinition{ex}{Example}
\newproof{proof}{Proof}
\newtheorem{cor}{Corollary}
\newtheorem{lem}{Lemma}

\begin{document}

\begin{frontmatter}



\title{Green's function for Sturm-Liouville problem}


\author[rvt]{O. Sh. Mukhtarov}
\ead{omukhtarov@yahoo.com} \cortext[cor1]{Corresponding Author (Tel:
+90 356 252 16 16, Fax: +90 356 252 15 85)}
\author[rvt]{K. Aydemir\corref{cor1}}
\ead{kadriye.aydemir@gop.edu.tr}

\address[rvt]{Department of Mathematics, Faculty of Arts and Science, Gaziosmanpa\c{s}a University,\\
 60250 Tokat, Turkey}

\begin{abstract}
The purpose of this study is to investigate a new class of boundary
value transmission problems (BVTP's) for Sturm-Liouville equation on
two separate intervals. We introduce modified inner product in
direct sum space $L_{2}[a,c)\oplus L_{2}(c,b]\oplus\mathbb{C}^{2}$
and define symmetric linear operator  in it such a way that the
considered problem  can be interpreted as an eigenvalue problem of
this operator. Then by suggesting an own approaches we construct
Green's function for problem under consideration and find the
resolvent function for corresponding inhomogeneous problem.
\end{abstract}


\end{frontmatter}


\section{Introduction}

Many interesting applications of Sturm-Liouville  theory arise in
quantum mechanics. Boundary value problems can be investigate also
through the methods of Green's function and eigenfunction expansion.
The main tool for solvability analysis of such problems is the
concept of Green's function. The concept of Green's functions is
very close to physical intuition (see \cite{du}). If one knows the
Green's function of a problem one can write down its solution in
closed form as linear combinations of integrals involving the
Green's function and the functions appearing in the inhomogeneities.
Green's functions can often be found in an explicit way, and in
these cases it is very efficient to solve the problem in this way.
Determination of Green's functions is also possible using
Sturm-Liouville theory. This leads to series representation of
Green's functions (see \cite{lev}).
\section{Statement of the problem}
In this study we shall investigate a new class  of BVP's which
consist of the Sturm-Liouville equation
\begin{equation}\label{1}
\mathcal{\ell}(y):=-p(x)y^{\prime \prime }(x)+ q(x)y(x)=\lambda y(x)
\end{equation}
to hold in finite interval $[a, b]$ except at one inner point $c \in
(a, b)$, where discontinuity in $u  \ \textrm{and} \ u'$   are
prescribed by  the transmission conditions at interior point $x=c$
\begin{equation}\label{4}
V_{j}(y):=\beta^{-}_{j1}y'(c-)+\beta^{-}_{j0}y(c-)+\beta^{+}_{j1}y'(c+)+\beta^{+}_{j0}y(c+)=0,
\ \ j=1,2
\end{equation}
together with eigenparameter-dependent boundary conditions at end
 points $x=a, b$
\begin{equation}\label{2}
V_{1}(y):=\alpha_{10}y(a)-\alpha_{11}y'(a)-\lambda(\alpha'_{10}y(a)-\alpha'_{11}y'(a))=0,
\end{equation}
\begin{equation}\label{3}
V_{2}(y):=\alpha_{20}y(b)-\alpha_{21}y'(b)+\lambda(\alpha'_{20}y(b)-\alpha'_{21}y'(b))=0,
\end{equation}
where $p(x)=p^{-}>0 \ \textrm{for} \ x \in [a, c) $, $p(x)=p^{+}>0 \
\textrm{for} \ x \in (c, b], $ the potential $q(x)$ is real-valued
function which continuous in each of the intervals $[a, c) \
\textrm{and} \ (c, b]$, and has a finite limits $q( c\mp0)$,
$\lambda$ \ is a complex spectral parameter, \ $\alpha_{ij}, \ \
\beta^{\pm}_{ij}, \ \alpha'_{ij} \ (i=1,2 \ \textrm{and} \ j=0,1)$
are real numbers. We want emphasize that the boundary value problem
studied here differs from the standard boundary value problems in
that it contains transmission conditions and the
eigenvalue-parameter appears not only in the differential equation,
but also in the boundary conditions. Moreover the coefficient
functions may have discontinuity at one interior point. Naturally,
eigenfunctions of this problem may have discontinuity at the one
inner point of the considered interval. The problems with
transmission conditions has become an important area of research in
recent years because of the needs of modern technology, engineering
and physics. Many of the mathematical problems encountered in the
study of boundary-value-transmission problem cannot be treated with
the usual techniques within the standard framework of  boundary
value problem (see \cite{ji}). Note that some special cases of this
problem arise after an application of the method of separation of
variables to a varied assortment of physical problems. For example,
some boundary value problems with transmission conditions arise in
heat and mass transfer problems \cite{lik}, in vibrating string
problems when the string loaded additionally with point masses
\cite{tik}, in diffraction problems \cite{voito}. Such properties,
as isomorphism, coerciveness with respect to the spectral parameter,
completeness and Abel bases of a system of root functions of the
similar boundary value problems with transmission conditions and its
applications to the corresponding initial boundary value problems
for parabolic equations have been investigated in \cite{osh2, osh3,
rasu1}. Also some problems with transmission conditions which arise
in mechanics (thermal conduction problems for a thin laminated
plate) were studied in  \cite{tite}.
\section{The ,,basic'' solutions and characteristic function}
 With a view to constructing the
characteristic function $\omega(\lambda )$  we shall define  two
basic solution $\varphi^{-}(x,\lambda) \ \textrm{and}  \
\psi^{-}(x,\lambda)$ on the left interval [a,c) and two basic
solution $\varphi^{+}(x,\lambda) \ \textrm{and}  \
\psi^{+}(x,\lambda)$ on the right interval (c,b] by special
procedure. Let $\varphi^{-}(x,\lambda) \ \textrm{and}  \
\psi^{+}(x,\lambda)$ be solutions of the equation $(\ref{1})$ on
[a,c) and (c,b] satisfying initial conditions
\begin{eqnarray}&&\label{7}
\varphi^{-}(a,\lambda)=\alpha _{11}-\lambda\alpha' _{11}, \
\frac{\partial\varphi^{-}(a,\lambda )}{\partial x})=\alpha
_{10}-\lambda\alpha' _{10} \\ &&\label{10}
\psi^{+}(b,\lambda)=\alpha _{21}+\lambda\alpha' _{21}, \
\frac{\partial \psi^{+}(b,\lambda)}{\partial x}=\alpha
_{20}+\lambda\alpha' _{20}
\end{eqnarray}
respectively. In terms of these solution we shall define the other
solutions $\varphi^{+}(x,\lambda) \ \textrm{and}  \
\psi^{-}(x,\lambda)$ by initial conditions
\begin{eqnarray}\label{8}
&&\varphi^{+}(c+,\lambda)
=\frac{1}{\Delta_{12}}(\Delta_{23}\varphi^{-}(c,\lambda
)+\Delta_{24}\frac{\partial\varphi^{-}(c,\lambda )}{\partial x})
\\ &&\label{9} \frac{\partial\varphi^{+}(c+,\lambda)}{\partial x}
=\frac{-1}{\Delta_{12}}(\Delta_{13}\varphi ^{-}(c,\lambda
)+\Delta_{14}\frac{\partial\varphi^{-}(c,\lambda )}{\partial x})
\\ \nonumber
\textrm{and}
 \\ &&\label{11}
\psi^{-}(c,\lambda )
=\frac{-1}{\Delta_{34}}(\Delta_{14}\psi^{+}(c,\lambda
)+\Delta_{24}\frac{\partial\psi^{+}(c,\lambda )}{\partial x})
\\ && \label{12} \frac{\partial\psi^{-}(c,\lambda )}{\partial x})
=\frac{1}{\Delta_{34}}(\Delta_{13}\psi ^{+}(c,\lambda
)+\Delta_{23}\frac{\partial\psi^{+}(c,\lambda )}{\partial x})
\end{eqnarray}
respectively,  where $\Delta_{ij} \ (1\leq i< j \leq 4)$ denotes the
determinant of the i-th
 and j-th columns of the matrix $$ T= \left[%
\begin{array}{cccc}
  \beta^{-}_{10} & \beta^{-}_{11} & \beta^{+}_{10} & \beta^{+}_{11} \\
  \beta^{-}_{20} & \beta^{-}_{21} & \beta^{+}_{20} & \beta^{+}_{21}
  \\
\end{array} %
 \right]. $$ The existence and uniqueness of these solutions are
follows from well-known theorem of ordinary differential equation
theory. Moreover by applying the method of \cite{ka} we can prove
that each of these solutions are entire functions of parameter
$\lambda \in \mathbb{C}$ for each fixed $x$. Taking into account
(\ref{8})-(\ref{12}) and the fact that the Wronskians $
\omega\pm(\lambda):=W[\varphi^{\pm}(x,\lambda ),\psi ^{\pm
}x,\lambda )]$ are independent of variable $x$  we have
\begin{eqnarray*}
\omega^{+}(\lambda ) &=&\varphi^{+}(c,\lambda )\frac{\partial\psi
^{+}(c,\lambda )}{\partial x}-\frac{\partial\varphi
^{+}(c,\lambda )}{\partial x}\psi ^{+}(c,\lambda ) \\
 &=&\frac{\Delta_{34}}{\Delta_{12}}(\varphi^{-}(c,\lambda )\frac{\partial\psi
^{-}(c,\lambda )}{\partial x}-\frac{\partial\varphi
^{-}(c,\lambda )}{\partial x}\psi ^{-}(c,\lambda )) \\
&=&\frac{\Delta_{34}}{\Delta_{12}} \omega^{-}(\lambda ).
\end{eqnarray*}
It is convenient to define  the characteristic function  \
$\omega(\lambda)$ for our problem $(\ref{1})-(\ref{3})$ as
$$
\omega(\lambda):= \Delta_{34} \omega^{-}(\lambda) = \Delta_{12} \
\omega^{+}(\lambda).
$$
Obviously, $\omega(\lambda)$ is an entire function. By applying the
technique of \cite{osh6} we can prove that there are infinitely many
eigenvalues $\lambda_{n}, \ n=1,2,...$ of the problem
$(\ref{1})-(\ref{3})$ which are coincide with the zeros of
characteristic function  \ $\omega(\lambda)$.
\begin{thm}\label{thm1}
Each eigenvalue of the problem (\ref{1})-(\ref{3}) is the simple
zero of \ $w(\lambda)$. \end{thm}
\begin{proof}
\end{proof}
\begin{lem}\label{bag}
Let $\lambda_{0}$  be zero of $w(\lambda)$. Then the solutions
$\varphi(x,\lambda_{0}) \ \textrm{and} \ \psi(x,\lambda_{0})$ are
linearly dependent.
\end{lem}
\begin{proof}
\end{proof}

\section{\textbf{Operator treatment in modified Hilbert space }}
To analyze the spectrum of the BVTP $(\ref{1})-(\ref{3})$ we shall
construct an adequate  Hilbert space and define a symmetric linear
operator in it  such a way that the considered problem can be
interpreted as the eigenvalue problem of this operator. For this we
assume that $$\Delta_{12}>0,  \ \Delta_{34}>0, \ \theta_{1}= \left[%
\begin{array}{cccc}
  \alpha_{11} & \alpha_{10}  \\
  \alpha'_{11} & \alpha'_{10}
  \\
\end{array}%
 \right]>0,  \  \theta_{2}= \left[%
\begin{array}{cccc}
  \alpha_{21} & \alpha_{20}  \\
  \alpha'_{21} & \alpha'_{20}
  \\
\end{array} %
 \right]>0 $$
and introduce modified inner products on direct sum space
$\mathcal{H}_{1}= L_{2}[a,c)\oplus L_{2}(c,b] \ \textrm{and} \
\mathcal{H}=\mathcal{H}_{1}\oplus \mathbb{C}^{2}$ by
\begin{eqnarray}
&&[f,g]_{\mathcal{H}_{1}}:=\frac{\Delta_{12}}{p^{-}} \ \int_{a}^{c-}
f(x)\overline{g(x)}dx + \frac{\Delta_{34}}{p^{+}} \int_{c+}^{b}
f(x)\overline{g(x)}dx\\ \nonumber \textrm{and} \\ &&\label{ic}
[F,G]_{\mathcal{H}}:=[f,g]_{\mathcal{H}_{1}} +
\frac{\Delta_{12}}{p^{-}\theta_{1}}f_{1}\overline{g_{1}}+\frac{\Delta_{34}}{p^{+}\theta_{2}}f_{2}\overline{g_{2}}
\end{eqnarray}
 for $F=\left(
  \begin{array}{c}
   f(x),
    f_{1}, f_{2} \\
  \end{array}
\right)$,\  $G=\left(
  \begin{array}{c}
    g(x),
  g_{1}, g_{2}  \\
  \end{array}
\right)$ $\in \mathcal{H}$ respectively. Obviously, these inner
products are equivalent to the standard inner products, so,
 $(\mathcal{H}, [.,.]_{\mathcal{H}})$
and $(\mathcal{H}_{1}, [.,.]_{\mathcal{H}_{1}})$ are also Hilbert
spaces. Let us now define the boundary functionals
\begin{eqnarray*}&&B_{a}[f]:= \alpha_{10}f(a)-\alpha_{11}f'(a),  \ \  B'_{a}[f]:=
\alpha'_{10}f(a)-\alpha'_{11}f'(a)\\
&&B_{b}[f]:= \alpha_{20}f(b)-\alpha_{21}f'(b),  \ \  B'_{b}[f]:=
\alpha'_{20}f(b)-\alpha'_{21}f'(b)
\end{eqnarray*}
and construct the operator $\mathcal{L}:\mathcal{H}\rightarrow
\mathcal{H}$   with the domain
\begin{eqnarray*}\label{2.3}
dom(\mathcal{L}):= &&\bigg \{F=(f(x), f_{1}, f_{2}):f(x), f'(x) \
\in AC_{loc}(a, c)\cap AC_{loc}(c, b),\nonumber\\ &&\textrm{and has
a finite limits} \ f(c\mp0) \ \textrm{and} \ f'(c\mp0), \ \ell F \in
L_{2}[a,b],\nonumber\\&&V_{3}(f)=V_{4}(f)=0,\ f_{1}= B'_{a}[f],
f_{2}=-B'_{b}[f]\bigg \}
\end{eqnarray*}
and action low$$ \mathcal{L}(f(x),B'_{a}[f], -B'_{b}[f])=(\ell f,
B_{a}[f], B_{b}[f]).$$ Then the problem $(\ref{1})-(\ref{3})$ can be
written in the operator equation form as
$$\mathcal{L}F=\lambda F, \ \ F=(f(x), B'_{a}[f], -B'_{b}[f])
 \in dom(\mathcal{L})$$
in the Hilbert space $\mathcal{H}$.
\begin{thm}\label{2.1}
The linear operator $\mathcal{L}$  is symmetric.
\end{thm}
\begin{proof}
By applying the method of  \cite{osh6} it is not difficult to show
that $dom(\mathcal{L})$ is dense in the Hilbert space $\mathcal{H}$.
Now let $F=(f(x),B'_{a}[f], -B'_{b}[f]),G=(g(x),B'_{a}[g],
-B'_{b}[g]) \in dom(\mathcal{L}).$ By partial integration we have
\begin{eqnarray}\label{2.2}
&&[\mathcal{L}F,G]_{\mathcal{H}}-[F,\mathcal{L}G]_{\mathcal{H}} =
 \Delta_{12} \ W(f, \overline{g};c-) -
\Delta_{12} \ W(f, \overline{g};a)  \nonumber\\
&&+\Delta_{34} \ W(f,\overline{g};b) - \Delta_{34} \
W(f,\overline{g};c+)
+\frac{\Delta_{12}}{p^{-}\theta_{1}}(B_{a}[f]\overline{B'_{a}[g]}-B'_{a}[f]\overline{B_{a}[g]})\nonumber\\
&&+
\frac{\Delta_{34}}{p^{+}\theta_{2}}(B'_{b}[f]\overline{B_{b}[g]}-B_{b}[f]\overline{B'_{b}[g]})
\end{eqnarray}
where, as usual, $W(f, \overline{g};x)$ denotes the Wronskians of
the functions $f$ \textrm{and} $ \overline{g}$. From the definitions
of boundary functionals we get that
\begin{eqnarray}\label{cl} &&B_{a}[f]\overline{B'_{a}[g]}-B'_{a}[f]\overline{B_{a}[g]}=p^{-}\theta_{1}
W(f,\overline{g};a), \\ && \label{c2}
B'_{b}[f]\overline{B_{b}[g]}-B_{b}[f]\overline{B'_{b}[g]}=-p^{+}\theta_{2}
W(f,\overline{g};b)
\end{eqnarray}
Further, taking in view the definition of $\mathcal{L}$ and initial
conditions $(\ref{7})-(\ref{12})$ we derive that
\begin{eqnarray}\label{c3}W(f, \overline{g};c-) =\frac{ \Delta_{34}}{\Delta_{12}} \
W(f, \overline{g};c+).
\end{eqnarray}
Finally, substituting  (\ref{cl}),  (\ref{c2}) and (\ref{c3}) in
(\ref{2.2}) we have
\[[\mathcal{L}F,G]_{\mathcal{H}}=[F,\mathcal{L}G]_{\mathcal{H}} \  \textrm{for \ every} \  F,G \in dom(\mathcal{L}), \] so the operator $\mathcal{L}$ is
symmetric in $\mathcal{H}$. The proof is complete.
\end{proof}
\begin{cor}\label{rem2} \textbf{(i)} The eigenvalues of the problem $(\ref{1})-(\ref{3})$ are real.\\
\textbf{(ii)} If $f(x)$ \textrm{and} $g(x)$   are eigenfunctions
corresponding to distinct eigenvalues, then they are ,,orthogonal"
in the sense of
\begin{eqnarray}\label{2.3}
[f,g]_{\mathcal{H}_{1}} +
\frac{\Delta_{12}}{p^{-}\theta_{1}}{}B'_{a}[f]B'_{a}[g]
+\frac{\Delta_{34}}{p^{+}\theta_{2}}{}B'_{b}[f]B'_{b}[g] =0.
\end{eqnarray}
where $F=(f(x),B'_{a}[f], -B'_{b}[f]),G=(g(x),B'_{a}[g], -B'_{b}[g])
\in dom(\mathcal{L})$.
\end{cor}
\begin{thm}\label{2.1}
The linear operator $\mathcal{L}$  is self-adjoint.
\end{thm}
\begin{proof}
\end{proof}
\section{\textbf{Solvability of the corresponding  inhomogeneous problem }}
 Now let $\lambda\in\mathbb{C}$ not be an eigenvalue of $\mathcal{L}$ and
consider the operator equation
\begin{eqnarray}\label{141u}
(\lambda I-\mathcal{L})Y=U,
\end{eqnarray}
for arbitrary \ $U=(u(x),u_{1},u_{2}) \in \mathcal{H}$. This
operator equation is equivalent to the following inhomogeneous BVTP
\begin{eqnarray}\label{141}
&& \ \ \ \ \ \ \ \ \ \ \ \ \ \ \ (\lambda -\ell )y(x)= u(x), \ \ x
\in [a,c)\cup(c,b]
\\ && \label{141v} V_{3}(y)=V_{4}(y)=0, \ \ \lambda
B'_{a}[y]-B_{a}[y]= u_{1}, \ \ -\lambda B'_{b}[y]-B_{b}[y]=u_{2}
\end{eqnarray}
 We shall search the resolvent function of this BVTP in the form
\begin{equation}
Y(x,\lambda )=\left\{
\begin{array}{c}
d_{11}(x,\lambda )\varphi^{-}(x,\lambda)+d_{12}(x,\lambda )\psi^{-}(x,\lambda)%
\hbox{ for }x\in \left[ a,c\right) \\
d_{21}(x,\lambda )\varphi^{+}(x,\lambda)+d_{22}(x,\lambda )\psi^{+}(x,\lambda)%
\hbox{ for }x\in \left( c,b\right]%
\end{array}
\right.  \label{(6.6)}
\end{equation}
where the functions $d_{11}(x,\lambda )$, $d_{12}(x,\lambda )$ and $%
d_{21}(x,\lambda )$, $d_{22}(x,\lambda )$ are the solutions of the
system of equations Since $\lambda $ is not an eigenvalue $
\omega^\pm(\lambda) \neq 0 $. By using the conditions $(\ref{141v})$
we can derive that
\begin{eqnarray*}\label{19} &&h_{12}(\lambda)=\frac{u_{1}}{\omega^{-}(\lambda)}, \
h_{21}(\lambda)=\frac{u_{2}}{\omega^{+}(\lambda)},\\ \label{20}
&&h_{11}(\lambda)=\frac{1}{p^{+}\omega^{+}(\lambda)}\int_{c+}^{b}\psi^{+}(y,\lambda)u(y)dy+\frac{u_{2}}{\omega^{+}(\lambda)}
\\ \nonumber
\textrm{and} \\ &&\label{21}
h_{22}(\lambda)=\frac{1}{p^{-}\omega^{-}(\lambda)}\int_{a}^{c-}\varphi^{-}(y,\lambda)u(y)dy+\frac{u_{1}}{\omega^{-}(\lambda)}.
\end{eqnarray*}
Thus
\begin{eqnarray}\label{22}
Y(x,\lambda)=\left\{\begin{array}{c}
               \frac{\Delta_{34}\psi^{-}(x,\lambda)}{p^{-}\omega(\lambda)}\int_{a}^{x}\varphi^{-}(y,\lambda)u(y)dy +
\frac{\Delta_{34}\varphi^{-}(x,\lambda)}{p^{-}\omega(\lambda)}\int_{x}^{c-}\psi^{-}(y,\lambda)u(y)dy \\
+\frac{\Delta_{12}\varphi^{-}(x,\lambda)}{\omega(\lambda)}(\frac{1}{p^{+}}\int_{c+}^{b}\psi^{+}(y,\lambda)u(y)dy+u_{2})+\frac{\Delta_{34}u_{1}
\psi^{-}(x,\lambda)}{\omega(\lambda)}
 \ \ \ \ \ for \  x \in [a,c) \\
                \\
\frac{\Delta_{12}\psi^{+}(x,\lambda)}{p^{+}\omega(\lambda)}\int_{c+}^{x}\varphi^{+}(y,\lambda)u(y)dy
+\frac{\Delta_{12}\varphi^{+}(x,\lambda)}{p^{+}\omega(\lambda)}\int_{x}^{b}\psi^{+}(y,\lambda)u(y)dy\\
+
\frac{\Delta_{34}\psi^{+}(x,\lambda)}{\omega(\lambda)}(\frac{1}{p^{-}}\int_{a}^{c-}\varphi^{-}(y,\lambda)u(y)dy+u_{1})+\frac{\Delta_{12}u_{2}
\varphi^{+}(x,\lambda)}{\omega(\lambda)}\
\ \ \ \ \ for \  x \in (c,b] \\
             \end{array}\right.
\end{eqnarray}
Let us introduce the Green's function as
\begin{eqnarray}\label{23}
G_{1}(x,y;\lambda)=\left\{\begin{array}{c}
\frac{\varphi^{-}(x,\lambda)\psi^{-}(y,\lambda)}{\Delta_{34}\omega^{-}(\lambda)},
\ \ \ if \ x \in[a,c), \  \ y \in [a,x)\ \  \\ \\

\frac{\psi^{-}(x,\lambda)\varphi^{-}(y,\lambda)}{\Delta_{34}\omega^{-}(\lambda)},
\ \ \ if \ x \in[a,c), \  \ y  \in [x,c)\ \  \\ \\

\frac{\psi^{-}(x,\lambda)\varphi^{+}(y,\lambda)}{\Delta_{34}\omega^{-}(\lambda)}, \ \ \ \ \ if \ x \in[a,c), \  \ y \in (c,b]\ \    \\

\frac{\varphi^{+}(x,\lambda)\psi^{-}(y,\lambda)}{\Delta_{12}\omega^{+}(\lambda)},
\ \ \ if \ x \in(c,b], \ \ y \in [a,c)\ \   \\ \\

 \frac{\varphi^{+}(x,\lambda)\psi^{+}(y,\lambda)}{\Delta_{12}\omega^{+}(\lambda)}, \ \ \ if \ \ x \in(c,b], \ \ y \in (c,x]\ \
 \\ \\

\frac{\psi^{+}(x,\lambda)
\varphi^{+}(y,\lambda)}{\Delta_{12}\omega^{+}(\lambda)}, \ \ \ if \
\ x \in(c,b], \ \ y \in [x,b]\ \   \\ \\
  \end{array}\right.
\end{eqnarray}
Then from (\ref{22}) and (\ref{23}) it follows that the considered
problem (\ref{141}), (\ref{141v}) has an unique solution given by
\begin{eqnarray}\label{2.16}
Y(x,\lambda)&=& \Delta_{34}\int_{a}^{c-} G_{1}(x,y;\lambda)u(y)dy+
 \Delta_{12}\int_{c+}^{b} G_{1}(x,y;\lambda)u(y)dy
\nonumber\\&+&
\Delta_{34}u_{1}\frac{\psi(x,\lambda)}{\omega(\lambda)}+\Delta_{12}u_{2}\frac{\varphi(x,\lambda)}{\omega(\lambda)}
\end{eqnarray}
\begin{cor}\label{corn}The resolvent operator can be represented as
\begin{eqnarray*}\label{141ö}
(\lambda I-\mathcal{L})^{-1}U(x)=\left\{\begin{array}{c}
 \int\limits_{a}^{b}G_{1}(x,y;\lambda )u(y)dy+\Delta_{34}u_{1}\frac{\psi(x,\lambda)}{\omega(\lambda)}+\Delta_{12}u_{2}\frac{\varphi(x,\lambda)}{\omega(\lambda)}\\
                \\
 \ \  B'_{a}[u]
 \\
-B'_{b}[u] \\
             \end{array} \right\}
\end{eqnarray*}
\end{cor}
\begin{thm}
 The resolvent operator $R(\lambda,A)$  is compact.
\end{thm}
\begin{proof}
\end{proof}
\begin{thm}\label{th3.1}(i) The modified Parseval equality
\begin{eqnarray}\label{a1}
\Delta_{12}  \ \int_{a}^{c} f^{2}(x)dx + \Delta_{34} \int_{c}^{b}
f^{2}(x)dx&=&\sum_{n=0}^{\infty}\mid
\Delta_{12} \ \int_{a}^{c} f(x)\psi_{n}(x)dx \nonumber\\
&+& \Delta_{34} \int_{c}^{b}f(x)\psi_{n}(x)dx\mid^{2}
\end{eqnarray}
is  hold for each $f \in L_{2}[a,c] \oplus L_{2}[c,b].$

\end{thm}

\begin{proof}
\end{proof}

\begin{thm}\label{thm3.2} Let $(f(x), (f)_{\beta}^{'}) \in
D(A)$. Then
\begin{eqnarray}\label{3.16} (i) \ \
 f(x)&=& \sum_{n=0}^{\infty}\big(\Delta_{12}  \
\int_{a}^{c} f(x)\psi_{n}(x)dx  + \Delta_{34}
\int_{c}^{b}f(x)\psi_{n}(x)dx\nonumber\\
&+&\frac{\Delta_{12}}{p^{-}\theta_{1}}f_{1}\overline{g_{1}}(\psi_{n})_{\beta}^{'}
\big)\psi_{n}(x)+\frac{\Delta_{34}}{p^{+}\theta_{2}}f_{2}\overline{g_{2}}
(\psi_{n})_{\beta}^{'} \big)\psi_{n}(x) \end{eqnarray} where, the
series converges absolutely and uniformly in whole $[a,c)\cup(c,b].$
(ii) The series (\ref{3.16}) may also be differentiated, the
differentiated series also being absolutely and uniformly convergent
in whole $[a,c)\cup(c,b].$
 \end{thm}
 \begin{proof}
\end{proof}

 \textbf{Example.} Consider the following simple case of the BVTP's
$(\ref{1})-(\ref{3})$
\begin{equation}\label{ex}
-y^{\prime \prime }(x)=\lambda y(x)
\end{equation}
\begin{equation}\label{ex1}
y(-1)+\lambda y'(-1))=0,
\end{equation}
\begin{equation}\label{ex2}
\lambda y(1)+y'(1))=0,
\end{equation}
\begin{equation}\label{ex3}
y'(0-)=y(+0), \  \ y'(-0)=2y'(+0)
\end{equation}\\
\begin{figure}
  \includegraphics[height=5cm,width=15cm]{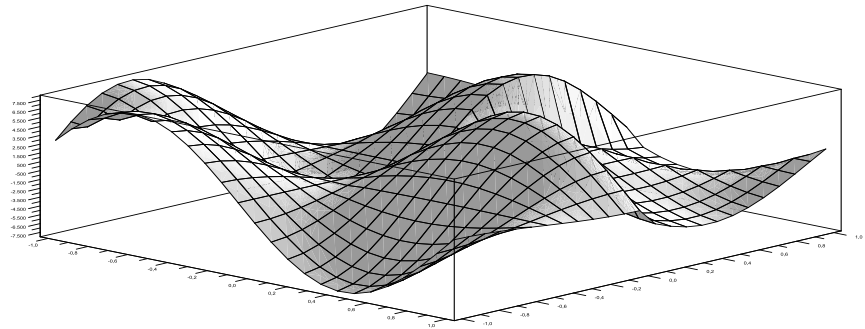}
  \caption{The graph of the Green's function $G(x,t,\mu)$ for $\mu=3$ }\label{f1}
\end{figure}
\begin{figure}
  \includegraphics[height=5cm,width=15cm]{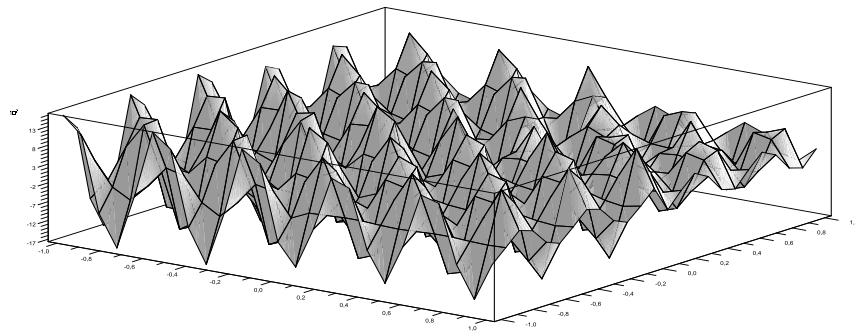}
  \caption{The graph of the Green's function $G(x,t,\mu)$ for $\mu=15$}\label{f2}
\end{figure}
The graph of the Green's function is displayed in Figure 1 and
Figure 2 for two different values of spectral parameter.

\end{document}